\documentclass[12pt]{amsart}
\usepackage{amssymb,latexsym,amsmath,amscd,amsthm,graphicx, color}
\usepackage[all]{xy}
\usepackage{pgf,tikz}
\usepackage{mathrsfs}
\usepackage{breqn}
\usepackage{amsmath}
\usepackage{fancyhdr}
\usetikzlibrary{arrows}
\definecolor{uuuuuu}{rgb}{0.26666666666666666,0.26666666666666666,0.26666666666666666}
\definecolor{xdxdff}{rgb}{0.49019607843137253,0.49019607843137253,1.}
\definecolor{ffqqqq}{rgb}{1.,0.,0.}

\newtheorem{theorem}{Theorem}[section]
\newtheorem{lemma}[subsection]{Lemma}

\theoremstyle{definition}
\newtheorem{remark}[subsection]{Remark}
\newtheorem{definition}[subsection]{Definition}

\newtheorem{corollary}[theorem]{Corollary}
\theoremstyle{remark}

\numberwithin{equation}{section}

\title[Optional running title]{Actual title}
\begin{document}
\title[Box Dimension of Mixed Katugampola Fractional Integral]{Box Dimension of Mixed Katugampola Fractional Integral of two-dimensional continuous functions }
%\title{Box Dimension of Mixed Katugampola Fractional Integral of two-dimensional continuous functions }

%
\author{Subhash Chandra, Syed Abbas}
\address{School of Basic Sciences, Indian Institute of Technology Mandi\\ Kamand (H.P.) - 175005, India}
\email{sahusubhash77@gmail.com; sabbas.iitk@gmail.com (Email of corresponding author)}
%\author{Syed Abbas}



%\thanks{The author was supported in part by UGC.}

%    General info
%\subjclass[2000]{Primary 54C40, 14E20; Secondary 46E25, 20C20}

%\date{January 1, 2001 and, in revised form, June 22, 2001.}

\begin{abstract}
The goal of this article is to study the box dimension of the mixed Katugampola fractional integral of two-dimensional continuous functions on $[0,1]\times [0,1]$. We prove that the box dimension of the mixed Katugampola fractional integral having fractional order $(\alpha=(\alpha_1,\alpha_2);~ \alpha_1>0,\alpha_2>0)$ of two-dimensional continuous functions on $[0,1]\times [0,1]$ is still two. Moreover, the results are also established for the mixed Hadamard fractional integral. 
\end{abstract}

\maketitle

%%%%%%%%%%%%%%%%%%%%%%%%%%%%%%%%%%%%%%%%%%%%%%%%%%%%%%%%%%%%%%%%%%%%%%%%

%%%%%%%%%%%%%%%%%%%%%%%%%%%%%%%%%%%%%%%%%%%%%%%%%%%%%%%%%%%%%%%%%%%%%%%%

\section{\textbf{Introduction}} Fractional calculus (FC) and fractal geometry (FG) have become rapidly growing fields in theory as well as applications. Since random fractals are a better example of a highly irregular functions, fractional calculus is the best mathematical operator for analyzing such a functions. Tatom gave a general relation between fractional calculus and fractal functions in \cite{FB}. FG is more prominent than classical geometry to study irregular sets. Theory on FG is given in reference \cite{MF}.  For various notions and definitions of fractional integrals and derivatives reader may refer to \cite{KIL,Samko}.\\ The box dimension plays an important role to study the smoothness of any irregular function. A connection between FC and fractal dimensions can be found in \cite{KIM,L2,Liang11,L3,WU2}. Liang \cite{L1} proved that the box dimension of a function which is of bounded variation and continuous on $[0,1]$ is $1$ and also box dimension of its Riemann-Livouville fractional  $I^{\nu} g$ on $[0,1]$ is $1$, where $$I^\nu g(x)= \frac{1}{\Gamma(\nu)}\int_{0}^x(x-s)^{\nu-1}g(s)ds.$$ Liang \cite{Liang11} investigated the fractal dimension of the fractional integral of Riemann-Liouville type of continuous function having box dimension $1$. We try to establish more general results for the  fractional integral of mixed Katugampola type. Box dimension of a  bivariate function which is of bounded variation in  Arzel\'{a} sense and  continuous on $[a,b]\times[c,d]$ has been investigated in \cite{V1}. It has been shown that box dimension of a bivariate function which is of bounded variation in Arzel\'{a} sense and continuous on  $[a,b]\times[c,d]$ is $2$ . Also some examples of $2$-dimensional functions are given which are not of bounded variation. Additionally, they have proved that the box dimension of the fractional integral of mixed Riemann-Liouville type of a continuous function which is of bounded variation in Arzel\'{a} sense  is $2.$ Analogous results can be seen for the mixed Katugampola fractional integral $ (\mathfrak{I}^{\alpha}f)$ in \cite{V2}, where
 $$ (\mathfrak{I}^{\alpha}f)(x,y)=\frac{(\rho_1+1)^{1-\alpha_1}(\rho_2+1)^{1-\alpha_2}}{\Gamma (\alpha_1)  \Gamma (\alpha_2)} \int_a ^x \int_c ^y (x^{\rho_1+1}-s^{\rho_1+1})^{\alpha_1-1} (y^{\rho_2+1}-t^{\rho_2+1})^{\alpha_2-1}s^{\rho_1}t^{\rho_2}f(s,t)dsdt. $$
The fractional integral of mixed Katugampola type is the unification of the fractional integral of mixed Riemann-Liouville type and fractional integral of mixed Hadamard type. Above results are based on analytical aspects in the sense that they have used bounded variation property in Arzel\'{a} sense to establish their results. In this work we do not require this condition. We have seen that bounded variation property plays an important role in case of bivariate functions. We can reinforce that deduction from bivarite to univariate case is easy but converse is not easy because for dealing with bivariate case we need more notions of bounded variation such as Arzel\'{a}, Peirpont and Hahn, for more details, we refer to \cite{CK3}. Sometimes it is difficult to maintain continuity in bivariate case.
 From above discussion, a natural question arises that what would be the box dimension of the mixed Katugampola fractional integral of a $2$-dimensional continuous function.\\
 In this article, we will prove that the box dimension of the mixed Katugampola fractional integral of $2$-dimensional continuous functions is also $2$. Furthermore, we will give the similar  result for the mixed Hadamard fractional integral.
\section{\textbf{Preliminaries}}
In this section, we give definition of fractional integrals and required terminologies related to this report.

\textit{2.1. Mixed Katugampola Fractional Integral:}
\begin{definition}
 Let $f$ be a function which is defined on a closed rectangle $[a,b] \times [c,d]$ and $a\geq 0,c\geq 0.$ Assuming that the following integral exists, mixed Katugampola fractional integral of $f$ is defined by $$ (\mathfrak{I}^{\alpha}f)(x,y)=\frac{(\rho_1+1)^{1-\alpha_1}(\rho_2+1)^{1-\alpha_2}}{\Gamma (\alpha_1)  \Gamma (\alpha_2)} \int_a ^x \int_c ^y (x^{\rho_1+1}-s^{\rho_1+1})^{\alpha_1-1} (y^{\rho_2+1}-t^{\rho_2+1})^{\alpha_2-1}s^{\rho_1}t^{\rho_2}f(s,t)dsdt ,$$ where $\alpha = ( \alpha_1, \alpha_2 )$ with $ \alpha_1 >0 , \alpha_2 >0$ and $(\rho_1,\rho_2)\neq(-1,-1).$
\end{definition}
\textit{2.2. Mixed Hadamard Fractional Integral:}\\\\
By using L'hospital rule and $\rho_1,\rho_2 \to -1^+$, mixed Katugampola reduces to the mixed Hadamard fractional integral.
\begin{definition}\label{Def2}
 Let $f$ be a function which is defined on a closed rectangle $[a,b] \times [c,d]$ and $a\geq 0,c\geq 0.$ Assuming that the following integral exists, mixed Hadamard fractional integral of $f$ is defined by $$ (\mathfrak{I}^{\gamma}f)(x,y)=\frac{1}{\Gamma (\gamma_1)  \Gamma (\gamma_2)} \int_a ^x \int_c ^y (\log\frac{x}{u})^{\gamma_1-1} (\log\frac{y}{v})^{\gamma_2-1}\frac{f(u,v)}{uv} dudv ,$$ where $\gamma = ( \gamma_1, \gamma_2 )$ with $ \gamma_1 >0 , \gamma_2 >0.$
\end{definition}
Reader may recommended to \cite{V2} for the above definitions of fractional integrals.\\
\textit{2.3. Range of $f$:}
\begin{definition}
Let $A=[a,b]\times [c,d]$ be the rectangle. For a function $f:A\to \mathbb{R}$, the maximum range of $f$ over $A$ is given by $$ R_f[A]:=\sup_{(t,u),(x,y)\in A}\lvert f(t,u)-(x,y)\rvert.$$
\end{definition}
\textit{2.4. Box Dimension:}\\\begin{definition}
Let $S\neq \emptyset$ be a bounded subset of $\mathbb{R}^n$. Let the smallest number of sets having diameter at most $\delta$ is denoted by $N_{\delta}(S)$ which can cover $S.$ Then 
\begin{equation}
\underline{\dim}_B(S)=\mathop{\underline{\lim}}_{\delta \to 0}\frac{\log N_\delta(S)}{-\log\delta}~~~~~\text{(Lower box dimension)}
\end{equation}
and 
\begin{equation}
\overline{\dim}_B(S)=\overline{\lim_{\delta \to 0}}\frac{\log N_\delta(S)}{-\log\delta}~~~~~\text{(Upper box dimension)}.
\end{equation}
If $\underline{\dim}_B(S)=\overline{\dim}_B(S),$ common value is called  the box dimension of $S.$ That is,
\begin{equation*}
\dim_B(S)=\lim_{\delta \to 0}\frac{\log N_\delta(S)}{-\log\delta}.
\end{equation*}
\end{definition}
Let $C(I\times J)$ be the set of continuous functions, where $I\times J=[0,1]\times [0,1]$. Let $C$ be an absolutely positive constant and even in the same line at different occurrences it may have different values. Sometimes we use the term fractional integral of mixed Katugampola type in place of mixed Katugampola fractional integral.
\section{\textbf{Main Results}}
In this section, first we give following lemmas which act as prelude to our main theorem.
\begin{lemma}\label{lmm}
Let $f:[0,1]\times [0,1] \to \mathbb{R}$ is continuous and $0<\delta<1,~ \frac{1}{\delta} <m,n<1+\frac{1}{\delta}$ for some $m,n \in \mathbb{N}.$ If the number of $\delta$-cubes is denoted by  $N_{\delta}(Gr(f))$  which intersect the graph $Gr(f)$ of the function $f$, then
\begin{equation}
\sum_{j=1}^n \sum_{i=1}^m \max \left \{\frac{R_f[A_{ij}]}{\delta},1 \right \} \leq N_\delta(Gr(f))\leq 2mn+\frac{1}{\delta} \sum_{j=1}^n \sum_{i=1}^m R_f[A_{ij}],
\end{equation}
where $A_{ij}$ is the $(i,j)$-th cell corresponding to the net under consideration.
\end{lemma}
\begin{proof}
If $f(x,y)$ is continuous on $I\times J$, the number of cubes having side $\delta$ in the part above $A_{ij}$ which intersect $Gr(f,I\times J)$ is atleast $$\max \left \{\frac{R_f[A_{ij}]}{\delta},1 \right \}$$
and at most $$2+\frac{R_f[A_{ij}]}{\delta}.$$
By summing over all such parts we get the required result.
\end{proof}
\begin{lemma}
Let $f(x,y)\in C(I\times J)$ and $0<\alpha_1<1,~0<\alpha_2<1.$ If $h_1>0,h_2>0$ and $x+h_1\le 1,y+h_2\le 1,$ then
\begin{dmath*}
(\mathfrak{I}^{\alpha}f)(x+h_1,y+h_2)-(\mathfrak{I}^{\alpha}f)(x,y)=\\\frac{(\rho_1+1)^{-\alpha_1}(\rho_2+1)^{-\alpha_2}}{\Gamma (\alpha_1)  \Gamma (\alpha_2)} \int_0^1 \int_0^1 (1-s)^{\alpha_1-1}(1-t)^{\alpha_2-1}.\\\left[\left((x+h_1)^{\rho_1+1}\right)^{\alpha_1}\left((y+h_2)^{\rho_2+1}\right)^{\alpha_2}f\left((x+h_1)s^{\frac{1}{\rho_1+1}},(y+h_2)t^{\frac{1}{\rho_2+1}}\right)\\-\left(x^{\rho_1+1}\right)^{\alpha_1}\left(y^{\rho_2+1}\right)^{\alpha_2}f\left(xs^{\frac{1}{\rho_1+1}},yt^{\frac{1}{\rho_2+1}}\right)\right]dsdt.
\end{dmath*}
\end{lemma}
\begin{proof} By considering the conditions of the lemma,
\begin{dmath}\label{2.2}
(\mathfrak{I}^{\alpha}f)(x+h_1,y+h_2)-(\mathfrak{I}^{\alpha}f)(x,y)=\\\frac{(\rho_1+1)^{1-\alpha_1}(\rho_2+1)^{1-\alpha_2}}{\Gamma (\alpha_1)  \Gamma (\alpha_2)} \int_0 ^{x+h_1} \int_0 ^{y+h_2} ((x+h_1)^{\rho_1+1}-s^{\rho_1+1})^{\alpha_1-1} ((y+h_2)^{\rho_2+1}-t^{\rho_2+1})^{\alpha_2-1}.\\s^{\rho_1}t^{\rho_2}f(s,t)dsdt\\ -\frac{(\rho_1+1)^{1-\alpha_1}(\rho_2+1)^{1-\alpha_2}}{\Gamma (\alpha_1)  \Gamma (\alpha_2)} \int_0 ^x \int_0 ^y (x^{\rho_1+1}-s^{\rho_1+1})^{\alpha_1-1} (y^{\rho_2+1}-t^{\rho_2+1})^{\alpha_2-1}s^{\rho_1}t^{\rho_2}f(s,t)dsdt 
\end{dmath}
By applying the integral transform, let $$\left(\frac{s}{x+h_1}\right)^{\rho_1+1}=u,$$ and  $$\left(\frac{t}{y+h_2}\right)^{\rho_2+1}=v.$$ Then $$dsdt =|J|dudv,$$ where
\[ 
J=
\begin{bmatrix}
\frac{\partial s}{\partial u} & \frac{\partial s}{\partial v}\\
\frac{\partial t}{\partial u} & \frac{\partial t}{\partial v}
\end{bmatrix}
\]
$$=\frac{(x+h_1)^{\rho_1+1}(y+h_2)^{\rho_2+1}}{(\rho_1+1)(\rho_2+1)s^{\rho_1}t^{\rho_2}}.$$
Thus, we have
\begin{dmath*}
\frac{(\rho_1+1)^{1-\alpha_1}(\rho_2+1)^{1-\alpha_2}}{\Gamma (\alpha_1)  \Gamma (\alpha_2)} \int_0 ^{x+h_1} \int_0 ^{y+h_2} ((x+h_1)^{\rho_1+1}-s^{\rho_1+1})^{\alpha_1-1} ((y+h_2)^{\rho_2+1}-t^{\rho_2+1})^{\alpha_2-1}.\\s^{\rho_1}t^{\rho_2}f(s,t)dsdt\\=\frac{(\rho_1+1)^{-\alpha_1}(\rho_2+1)^{-\alpha_2}}{\Gamma (\alpha_1)  \Gamma (\alpha_2)} \int_0^1 \int_0^1 (1-s)^{\alpha_1-1}(1-t)^{\alpha_2-1} \left((x+h_1)^{\rho_1+1}\right)^{\alpha_1}\left((y+h_2)^{\rho_2+1}\right)^{\alpha_2}.\\f\left((x+h_1)s^{\frac{1}{\rho_1+1}},(y+h_2)t^{\frac{1}{\rho_2+1}}\right)dsdt.
\end{dmath*}
Similarly \begin{dmath*}
\frac{(\rho_1+1)^{1-\alpha_1}(\rho_2+1)^{1-\alpha_2}}{\Gamma (\alpha_1)  \Gamma (\alpha_2)} \int_0 ^x \int_0 ^y (x^{\rho_1+1}-s^{\rho_1+1})^{\alpha_1-1} (y^{\rho_2+1}-t^{\rho_2+1})^{\alpha_2-1}s^{\rho_1}t^{\rho_2}f(s,t)dsdt=\\\frac{(\rho_1+1)^{-\alpha_1}(\rho_2+1)^{-\alpha_2}}{\Gamma (\alpha_1)  \Gamma (\alpha_2)} \int_0^1 \int_0^1 (1-s)^{\alpha_1-1}(1-t)^{\alpha_2-1} \left(x^{\rho_1+1}\right)^{\alpha_1}\left(y^{\rho_2+1}\right)^{\alpha_2} f\left(xs^{\frac{1}{\rho_1+1}},yt^{\frac{1}{\rho_2+1}}\right)dsdt.
\end{dmath*}
Consequently, we get the desired result by using these two values in \ref{2.2}.
\end{proof}
%%%%%%%%%%%%%%%%%%%%%%%%%%%%%%%%%%%%%%%%%%%%%%%%%%%%%%%%
Now, we will establish our main result.
\begin{theorem}\label{Thmain}
Let a non-negative function $f(x,y)\in C(I\times J)$ and $0<\alpha_1<1,~0<\alpha_2<1, \rho_1<-1,~\rho_2<-1.$\\ If
 \begin{equation}\label{11}
\dim_B Gr(f, I\times J)=2,
\end{equation}
then, the box dimension of the mixed Katugampola fractional integral of $f(x,y)$ of order $\alpha=(\alpha_1,\alpha_2)$ exists and is equal to $2$ on $I\times J$, as
\begin{equation}\label{eq12}
\dim_B Gr(\mathfrak{I}^{\alpha}f,I\times J)=2.
\end{equation}
\end{theorem}
\begin{proof}
Since $f(x,y)\in C(I\times J)$, $(\mathfrak{I}^{\alpha}f)(x,y)$ is also continuous on $I\times J$ (from Theorem 3.4 in \cite{V2}). From the definition of the box dimension, we can get 
\begin{equation}\label{12}
\underline{\dim}_B Gr(\mathfrak{I}^{\alpha}f,I\times J) \geq 2.
\end{equation}
To prove Equation \ref{eq12}, we have to show \ref{13}
\begin{equation}\label{13}
\overline{\dim}_B Gr(\mathfrak{I}^{\alpha}f,I\times J) \leq 2.
\end{equation}
Suppose that $0<\delta<\frac{1}{2}$, $ \frac{1}{\delta} <m,n<1+\frac{1}{\delta}$ and  $N_{\delta}(Gr(f))$ is the number of $\delta$-cubes that intersect $Gr(f)$. From Equation \ref{11}, it holds
\begin{equation*}
\lim_{\delta \to 0} \frac{\log N_{\delta}(Gr(f))}{-\log\delta}=2.
\end{equation*}
Let $N_{\delta}(Gr(\mathfrak{I}^\alpha f))$ is the number of $\delta$-cubes that intersect $Gr(\mathfrak{I}^\alpha f)$. Thus Inequality \ref{13} can be written as
\begin{equation}\label{14}
\overline{\lim_{\delta \to 0} } \frac{\log N_{\delta}(Gr(\mathfrak{I}^\alpha f))}{-\log\delta} \leq 2.
\end{equation}
Now, we have to prove Inequality \ref{14}.\\ 
For $0<\delta<\frac{1}{2}$, $ \frac{1}{\delta} <m,n<1+\frac{1}{\delta},$   let non-negative integers $i$ and $j$ such that $0\le i \le m$,  $0\le j \le n$. Then
\begin{dmath*}
\left| \frac{(\rho_1+1)^{-\alpha_1})(\rho_2+1)^{-\alpha_2}}{\Gamma (\alpha_1)\Gamma (\alpha_2)} \right| R_{\mathfrak{I}^\alpha f}[A_{ij}]=\sup _{(x+h_1,y+h_2),(x,y)\in A_{ij}} \lvert (\mathfrak{I}^\alpha f)(x+h_1,y+h_2)-(\mathfrak{I}^\alpha f)(x,y)\rvert,
\end{dmath*}
where $A_{ij}=[i\delta,(i+1)\delta]\times [j\delta,(j+1)\delta].$\\
Here, 
\begin{dmath*}
=\lvert (\mathfrak{I}^\alpha f)(x+h_1,y+h_2)-(\mathfrak{I}^\alpha f)(x,y)\rvert=\left| \int_0^1 \int_0^1 (1-s)^{\alpha_1-1}(1-t)^{\alpha_2-1} \left((x+h_1)^{\rho_1+1}\right)^{\alpha_1}\left((y+h_2)^{\rho_2+1}\right)^{\alpha_2}.\\f\left((x+h_1)s^{\frac{1}{\rho_1+1}},(y+h_2)t^{\frac{1}{\rho_2+1}}\right)dsdt\\-\int_0^1 \int_0^1 (1-s)^{\alpha_1-1}(1-t)^{\alpha_2-1}\left((x+h_1)^{\rho_1+1}\right)^{\alpha_1}\left((y+h_2)^{\rho_2+1}\right)^{\alpha_2}f\left(xs^{\frac{1}{\rho_1+1}},yt^{\frac{1}{\rho_2+1}}\right)dsdt\\ +\int_0^1 \int_0^1 (1-s)^{\alpha_1-1}(1-t)^{\alpha_2-1} \left((x+h_1)^{\rho_1+1}\right)^{\alpha_1}\left((y+h_2)^{\rho_2+1}\right)^{\alpha_2}f\left(xs^{\frac{1}{\rho_1+1}},yt^{\frac{1}{\rho_2+1}}\right)dsdt\\
- \int_0^1 \int_0^1 (1-s)^{\alpha_1-1}(1-t)^{\alpha_2-1} \left(x^{\rho_1+1}\right)^{\alpha_1}\left(y^{\rho_2+1}\right)^{\alpha_2} f\left(xs^{\frac{1}{\rho_1+1}},yt^{\frac{1}{\rho_2+1}}\right)dsdt \right |\\
\le \left |\int_0^1 \int_0^1 (1-s)^{\alpha_1-1}(1-t)^{\alpha_2-1} \left((x+h_1)^{\rho_1+1}\right)^{\alpha_1}\left((y+h_2)^{\rho_2+1}\right)^{\alpha_2}. \left[f\left((x+h_1)s^{\frac{1}{\rho_1+1}},(y+h_2)t^{\frac{1}{\rho_2+1}}\right)-f\left(xs^{\frac{1}{\rho_1+1}},yt^{\frac{1}{\rho_2+1}}\right) \right] dsdt\right |\\
+\int_0^1 \int_0^1 (1-s)^{\alpha_1-1}(1-t)^{\alpha_2-1}f\left(xs^{\frac{1}{\rho_1+1}},yt^{\frac{1}{\rho_2+1}}\right).\\\left[ \left((x+h_1)^{\rho_1+1}\right)^{\alpha_1}\left((y+h_2)^{\rho_2+1}\right)^{\alpha_2}-\left(x^{\rho_1+1}\right)^{\alpha_1}\left(y^{\rho_2+1}\right)^{\alpha_2} \right]dsdt 
\leq \left((x+h_1)^{\rho_1+1}\right)^{\alpha_1}\left((y+h_2)^{\rho_2+1}\right)^{\alpha_2} \left |\int_0^1 \int_0^1 (1-s)^{\alpha_1-1}(1-t)^{\alpha_2-1}.\\ \left[f\left((x+h_1)s^{\frac{1}{\rho_1+1}},(y+h_2)t^{\frac{1}{\rho_2+1}}\right)-f\left(xs^{\frac{1}{\rho_1+1}},yt^{\frac{1}{\rho_2+1}}\right) \right] dsdt \right |\\
+\int_0^1 \int_0^1 (1-s)^{\alpha_1-1}(1-t)^{\alpha_2-1}f\left(xs^{\frac{1}{\rho_1+1}},yt^{\frac{1}{\rho_2+1}}\right).\\\left[ \left((x+h_1)^{\rho_1+1}\right)^{\alpha_1}\left((y+h_2)^{\rho_2+1}\right)^{\alpha_2}-\left(x^{\rho_1+1}\right)^{\alpha_1}\left(y^{\rho_2+1}\right)^{\alpha_2} \right]dsdt .
\end{dmath*}
Let $i\geq 1,j\geq 1.$ On the one hand,
\begin{dmath*}
\left |\int_0^1 \int_0^1 (1-s)^{\alpha_1-1}(1-t)^{\alpha_2-1} \left[f\left((x+h_1)s^{\frac{1}{\rho_1+1}},(y+h_2)t^{\frac{1}{\rho_2+1}}\right)-f\left(xs^{\frac{1}{\rho_1+1}},yt^{\frac{1}{\rho_2+1}}\right) \right] dsdt \right |\\
=\Bigl| \int_0^{\frac{1}{i+1}} \int_0^{\frac{1}{j+1}} (1-s)^{\alpha_1-1}(1-t)^{\alpha_2-1} \left[f\left((x+h_1)s^{\frac{1}{\rho_1+1}},(y+h_2)t^{\frac{1}{\rho_2+1}}\right)-f\left(xs^{\frac{1}{\rho_1+1}},yt^{\frac{1}{\rho_2+1}}\right) \right] dsdt \Bigr|\\
+\sum_{l=1}^j \Bigl| \int_0^{\frac{1}{i+1}} \int_{\frac{l}{j+1}}^{\frac{l+1}{j+1}} (1-s)^{\alpha_1-1}(1-t)^{\alpha_2-1}.\\ \left[f\left((x+h_1)s^{\frac{1}{\rho_1+1}},(y+h_2)t^{\frac{1}{\rho_2+1}}\right)-f\left(xs^{\frac{1}{\rho_1+1}},yt^{\frac{1}{\rho_2+1}}\right) \right] dsdt \Bigr|\\
+\sum_{r=1}^i\Bigl| \int_{\frac{r}{i+1}}^{\frac{r+1}{i+1}} \int_0^{\frac{1}{j+1}} (1-s)^{\alpha_1-1}(1-t)^{\alpha_2-1}.\\ \left[f\left((x+h_1)s^{\frac{1}{\rho_1+1}},(y+h_2)t^{\frac{1}{\rho_2+1}}\right)-f\left(xs^{\frac{1}{\rho_1+1}},yt^{\frac{1}{\rho_2+1}}\right) \right] dsdt \Bigr|\\
+\sum_{r=1}^i \sum_{l=1}^j\Bigl| \int_{\frac{r}{i+1}}^{\frac{r+1}{i+1}} \int_{\frac{l}{j+1}}^{\frac{l+1}{j+1}} (1-s)^{\alpha_1-1}(1-t)^{\alpha_2-1}.\\ \left[f\left((x+h_1)s^{\frac{1}{\rho_1+1}},(y+h_2)t^{\frac{1}{\rho_2+1}}\right)-f\left(xs^{\frac{1}{\rho_1+1}},yt^{\frac{1}{\rho_2+1}}\right) \right] dsdt \Bigr|
\leq \frac{1}{(i+1)(j+1)}R_f\left[[0,\delta]\times [0,\delta]\right]\\
+\sum_{l=1}^j \frac{1}{(i+1)(j+1)}\left( R_f\left[[0,\delta]\times [(l-1)\delta,l\delta]\right]+R_f\left[[0,\delta]\times [l\delta,(l+l)\delta]\right]\right)\\
+\sum_{r=1}^i\frac{1}{(i+1)(j+1)}(R_f\left[[(r-1)\delta,r\delta]\times [0,\delta]\right]+R_f\left[[r\delta,(r+1)\delta]\times [0,\delta]\right])\\
+\sum_{r=1}^i \sum_{l=1}^j\frac{1}{(i+1)(j+1)}(R_f\left[[(r-1)\delta,r\delta]\times [(l-1)\delta,l\delta]\right]+R_f\left[[(r-1)\delta,r\delta]\times [l\delta,(l+1)\delta]\right]\\+R_f\left[[r\delta,(r+1)\delta]\times [(l-1)\delta,l\delta]\right]+R_f\left[[r\delta,(r+1)\delta]\times [l\delta,(l+1)\delta]\right]).
\end{dmath*}
By using Bernoulli's inequality $(1+u)^{r'}\leq 1+r'u$ for $0\leq r' \leq 1$ and $u\geq -1$, we can see that $$\int_0^{\frac{1}{i+1}} \int_0^{\frac{1}{j+1}} (1-s)^{\alpha_1-1}(1-t)^{\alpha_2-1}dsdt \le \frac{1}{(i+1)(j+1)}. $$\\
On the other hand, for a suitable constant $C$, we have
\begin{dmath*}
\int_0^1 \int_0^1 (1-s)^{\alpha_1-1}(1-t)^{\alpha_2-1}f\left(xs^{\frac{1}{\rho_1+1}},yt^{\frac{1}{\rho_2+1}}\right)\left[ \left((x+h_1)^{\rho_1+1}\right)^{\alpha_1}\left((y+h_2)^{\rho_2+1}\right)^{\alpha_2}-\left(x^{\rho_1+1}\right)^{\alpha_1}\left(y^{\rho_2+1}\right)^{\alpha_2} \right]dsdt\\
\leq  \frac{ C \max _{0\leq (x,y) \leq 1} f(x,y)}{\alpha_1\alpha_2}.
\end{dmath*}
From Lemma \ref{lmm}, we have
\begin{dmath*}
N_\delta(Gr(\mathfrak{I}^\alpha f))\leq  2mn+\frac{1}{\delta} \sum_{j=1}^n \sum_{i=1}^m R_{\mathfrak{I}^\alpha f}[A_{ij}]
\leq 2mn+\frac{1}{\delta} \sum_{j=1}^n \sum_{i=1}^m \left(\frac{1}{(i+1)(j+1)}R_f\left[[0,\delta]\times [0,\delta]\right]\\
+\sum_{l=1}^j \frac{1}{(i+1)(j+1)}\left( R_f\left[[0,\delta]\times [(l-1)\delta,l\delta]\right]+R_f\left[[0,\delta]\times [l\delta,(l+l)\delta]\right]\right)\\
+\sum_{r=1}^i\frac{1}{(i+1)(j+1)}(R_f\left[[(r-1)\delta,r\delta]\times [0,\delta]\right]+R_f\left[[r\delta,(r+1)\delta]\times [0,\delta]\right])\\
+\sum_{r=1}^i \sum_{l=1}^j\frac{1}{(i+1)(j+1)}(R_f\left[[(r-1)\delta,r\delta]\times [(l-1)\delta,l\delta]\right]+R_f\left[[(r-1)\delta,r\delta]\times [l\delta,(l+1)\delta]\right]+R_f\left[[r\delta,(r+1)\delta]\times [(l-1)\delta,l\delta]\right]+R_f\left[[r\delta,(r+1)\delta]\times [l\delta,(l+1)\delta]\right]) \right)\\
+\frac{1}{\delta} \sum_{j=1}^n \sum_{i=1}^m \frac{ C \max _{0\leq (x,y) \leq 1} f(x,y)}{\alpha_1\alpha_2}
\leq \frac{1}{\delta}\left(C+ \sum_{j=1}^n \sum_{i=1}^m \left( \frac{1}{(i+1)(j+1)}R_f\left[[0,\delta]\times [0,\delta]\right]\\
+\sum_{l=1}^j \frac{1}{(i+1)(j+1)}\left( R_f\left[[0,\delta]\times [(l-1)\delta,l\delta]\right]+R_f\left[[0,\delta]\times [l\delta,(l+l)\delta]\right]\right)\\
+\sum_{r=1}^i\frac{1}{(i+1)(j+1)}(R_f\left[[(r-1)\delta,r\delta]\times [0,\delta]\right]+R_f\left[[r\delta,(r+1)\delta]\times [0,\delta]\right])\\
+\sum_{r=1}^i \sum_{l=1}^j\frac{1}{(i+1)(j+1)}(R_f\left[[(r-1)\delta,r\delta]\times [(l-1)\delta,l\delta]\right]+
R_f\left[[(r-1)\delta,r\delta]\times\\ [l\delta,(l+1)\delta]\right]+R_f\left[[r\delta,(r+1)\delta]\times [(l-1)\delta,l\delta]\right]+R_f\left[[r\delta,(r+1)\delta]\times [l\delta,(l+1)\delta]\right]) \right)\right)
\leq\frac{C}{\delta} \left( \sum_{j=0}^n \sum_{i=0}^m \frac{1}{(i+1)(j+1)} \right)\left( \sum_{j=1}^n \sum_{i=1}^m R_f[A_{ij}] \right)
\leq \frac{C}{\delta} (\log m) (\log n) \sum_{j=0}^n \sum_{i=0}^m R_f[A_{ij}] 
\leq C (\log m) (\log n) N_\delta(Gr(f)).
\end{dmath*}
Therefore, 
\begin{dmath*}
\frac{\log N_\delta(Gr(\mathfrak{I}^\alpha f))}{-\log \delta}\leq \frac{\log \left \{ C (\log m) (\log n) N_\delta(Gr(f))\right\}}{-\log \delta}
\leq \frac{\log C}{-\log \delta}+\frac{\log(\log m)}{-\log \delta}+\frac{\log(\log n)}{-\log \delta}+\frac{\log N_\delta(Gr(f))}{-\log \delta}.
\end{dmath*}
So, we obtain
\begin{dmath*}
\overline{\dim}_B Gr(\mathfrak{I}^{\alpha}f,I\times J)=\overline{\lim_{\delta \to 0}}\frac{\log N_\delta(Gr(\mathfrak{I}^\alpha f))}{-\log \delta}
\leq \overline{\lim_{\delta \to 0}} \left( \frac{\log C}{-\log \delta}+\frac{\log(\log m)}{-\log \delta}+\frac{\log(\log n)}{-\log \delta}+\frac{\log N_\delta(Gr(f))}{-\log \delta}\right)
\leq \overline{\lim_{\delta \to 0}}\frac{\log N_\delta(Gr(f))}{-\log \delta}
=\lim_{\delta \to 0}\frac{\log N_\delta(Gr(f))}{-\log \delta}=2.
\end{dmath*}
So, Inequality \ref{14} holds. From Inequalities \ref{12} and \ref{14}, we get  \ref{eq12}.
\end{proof}
\begin{corollary}
Let $0<\alpha_1<1,~0<\alpha_2<1$, ~$\rho_1<-1,~\rho_2<-1$ and $f$ is a continuous function of bounded variation on $[0,1]\times[0,1].$ Then $$\dim_BGr(\mathfrak{I}^{\alpha}f)=2.$$
\end{corollary}
\begin{proof}
From Lemma 3.7 in \cite{V2}, for a function $f$ which is continuous and of bounded variation in Arzel\'{a} sense on $[0,1]\times[0,1]$, we have $$\dim_BGr(f)=2.$$
So, from Theorem \ref{Thmain}, we obtain $$\dim_BGr(\mathfrak{I}^{\alpha}f)=2.$$
This completes the proof.
\end{proof}
\begin{remark}
In \cite{V2}, Verma and Viswanathan proved that the fractional integral of mixed Katugampola type of a bounded variation function is again bounded variation in Arzel\'{a} sense. By using this result they deduce that the box dimension of the fractional integral of mixed Katugampola type is $2$. Their results more on analytical aspects. But, we have proved that if $f$ is continuous function having box dimension $2$, then the box dimension of the fractional integral of mixed Katugampola type of $f$ is also $2.$ So, our results more on dimensional aspects.
\end{remark}
\begin{theorem}\label{thH}
Let a non-negative function $f(x,y)\in C(I\times J)$ and $0<\gamma_1<1,0<\gamma_2<1.$\\ If
 \begin{equation}
\dim_B Gr(f, I\times J)=2,
\end{equation}
then, the box dimension of the mixed Hadamard fractional integral of $f(x,y)$ of order $\gamma=(\gamma_1,\gamma_2)$ exists and is equal to $2$ on $I\times J$, as
\begin{equation}
\dim_B Gr(\mathfrak{I}^{\gamma}f,I\times J)=2.
\end{equation}
\end{theorem}
The proof of the Theorem \ref{thH} follows as Theorem \ref{Thmain}.
\begin{lemma}
\cite{V1} \label{lmV2} Suppose a continuous function $h:[c,d] \to \mathbb{R}$. Define a set $H=\{(x,y,h(y)):x\in [a,b], y \in[c,d]\}$ and $a<b.$ Then, $\overline{\dim}_B(H) \leq \overline{\dim}_B(Gr(h))+1.$
\end{lemma}
\begin{remark}\label{rmk1}
Let $g:[a,b] \to \mathbb{R}$ and $h:[c,d] \to \mathbb{R}$ are two continuous maps. Now, define $g_1,g_2:[a,b]\times [c,d]\to \mathbb{R}$ such that $$g_1(x,y)=g(x)+h(y), ~~\text{and}~~~g_2(x,y)=g(x)h(y).$$
By using Lemma \ref{lmV2}, we have $\overline{\dim}_BGr(g_1)\le \overline{\dim}_BGr(h)+1$ and $\overline{\dim}_BGr(g_2)\le \overline{\dim}_BGr(h)+1.$
\end{remark}
In the following remark, we corroborate the Theorem \ref{thH} and try to establish relation with univariate case.
\begin{remark}
Let $h:[a,b]\to \mathbb{R}$ be a continuous function having box dimension $1$. We define a bivariate continuous function $f:[a,b]\times [c,d] \to \mathbb{R}$ such that $f(x,y)=h(x).$  From \ref{Def2}, we have $$ (\mathfrak{I}^{\gamma}f)(x,y)=\frac{1}{\Gamma (\gamma_1)  \Gamma (\gamma_2)} \int_a ^x \int_c ^y (\log\frac{x}{u})^{\gamma_1-1} (\log\frac{y}{v})^{\gamma_2-1}\frac{f(u,v)}{uv} dudv.$$
For $\gamma_2=1,$ we get
$$ (\mathfrak{I}^{\gamma}f)(x,y)=\frac{1}{\Gamma (\gamma_1)} \int_a ^x \int_c ^y (\log\frac{x}{u})^{\gamma_1-1}\frac{f(u,v)}{uv} dudv.$$
From the definition of $f$, we obtain
$$ (\mathfrak{I}^{\gamma}f)(x,y)=\frac{\log(\frac{y}{c})}{\Gamma (\gamma_1)} \int_a ^x(\log\frac{x}{u})^{\gamma_1-1}\frac{h(u)}{u} du.$$
Now, we have the following relation between the  fractional integral of Hadamard type and  mixed Hadamard type  $$ (\mathfrak{I}^{\gamma}f)(x,y)=\log(\frac{y}{c})(\mathfrak{I}^{\gamma_1}h)(x),$$
where the Hadamard fractional integral is given by $$(\mathfrak{I}^{\gamma_1}h)(x)=\frac{1}{\Gamma (\gamma_1)} \int_a ^x(\log\frac{x}{u})^{\gamma_1-1}\frac{h(u)}{u}du.$$
From Remark \ref{rmk1}, we have $\overline{\dim}_BGr(\mathfrak{I}^{\gamma}f)\le \overline{\dim}_BGr(\mathfrak{I}^{\gamma_1}h)+1.$ Since, $\dim_BGr(h)=1,$ from \cite{JY} it follows that $\dim_BGr(\mathfrak{I}^{\gamma_1} h)=1$, and hence $\dim_BGr(\mathfrak{I}^\gamma f)=2.$ This corroborates Theorem \ref{thH}.
\end{remark}
\subsection*{Acknowledgements} The First author thanks to CSIR, India for the grant with file number 09/1058(0012)/2018-EMR-I.
%%%%%%%%%%%%%%%%%%%%%%%%%%%%%%%%%%%%%%%%%%%%
\bibliographystyle{amsplain}

\end{document}